\documentclass[11pt]{amsart}
\usepackage{amsmath,amssymb,amsthm, mathtools}
\usepackage[usenames,dvipsnames]{color}
\usepackage{hyperref}  
\usepackage{tikz-cd}

\newcommand\C{\mathbb{C}}
\newcommand\FPP{\textrm{FPP}}
\newenvironment{Aut}{{\rm Aut}}{}

\newtheorem{theorem}{Theorem}
\newtheorem{lemma}[theorem]{Lemma}
\newtheorem{proposition}[theorem]{Proposition} 
\theoremstyle{definition}
\newtheorem{question}{Question}

\begin{document}
\title[]{Corrigendum and addendum to ``Fixed-point-free elements of iterated monodromy groups" [Trans. of the American Mathematical Society Vol. 367, No. 3 (2015), 2023-2049]}
\author{Jorge Fariña-Asategui}
\author{Rafe Jones}
\author{Santiago Radi}

\begin{abstract}
In this note we document a gap in an argument in the above paper, and point to new work in the literature giving a complete proof of the main result.
\end{abstract}

\maketitle

We adopt the notation and definitions of \cite{fpfree}.  The argument proving Theorem 1.5 on \cite[p. 2033]{fpfree} contains a gap. The hypotheses of Theorem 1.5 are that $G \leq \Aut(X^*)$ is contracting, contains a spherically transitive element, and satisfies the property that every element of $\mathcal{N}_1$ fixes either no ends of $X^*$ or infinitely many. The argument given is sufficient to deduce that 
\begin{equation} \label{newcond}
\text{every element of $G$ fixes either no ends of $X^*$ or infinitely many.} 
\end{equation}
Denote by $G_\infty$ the closure of $G$ in $\Aut(X^*)$. Martingale methods imply that $\mu(\{g~\in~G_\infty  : \text{$g$ fixes infinitely many ends of $X^*$})\} = 0$, where $\mu$ is the Haar probability measure on the compact group $G_\infty$. The proof of Theorem 1.5 attempts to extend the dichotomy in \eqref{newcond} to elements of $G_\infty$, but this does not work because $G_\infty$ is not in general contracting even when $G$ is.  If \eqref{newcond} applied to $G_\infty$, then the martingale methods would give $\mu(\{g \in G_\infty  : \text{$g$ fixes at least one end of $X^*$})\} = 0$, which is the stated conclusion of Theorem 1.5.  In \cite{fpfree} this last quantity is denoted $\mathcal{F}(G)$, but in the subsequent literature the standard notation has become $\FPP(G)$, which we adopt here.

Theorem 1.1 of \cite{fpfree} states that if a post-critically finite $f \in \C[z]$ has degree at least 2 and iterated monodromy group $G$, and if $f$ is not exceptional, then $\FPP(G) = 0$. The proof relies on Theorem 1.5, and therefore the methods of \cite{fpfree} are sufficient only to prove the theorem with the statement that $\FPP(G) = 0$ replaced by \eqref{newcond}.  In Corollary 1.6 of \cite{fpfree}, which replaces the hypothesis on $\mathcal{N}_1$ in Theorem 1.5 with the assumption that $G$ is torsion-free, the conclusion should also be changed from $\FPP(G) = 0$ to the statement \eqref{newcond}.

A complete proof of Theorem 1.1 of \cite{fpfree} has been recently been given by the first and third authors \cite{fpp2}, building on methods introduced in \cite{fpp1}. The arguments of \cite{fpp1} and \cite{fpp2} take a different approach from those of \cite{fpfree}: rather than relying on the contracting property of iterated monodromy groups, they show that such groups are in general mixing groups (see \cite[Definition 2.2]{fpp2}), a strong group-theoretic property. The mixing condition allows one to apply martingale methods to obtain $\FPP(G) = 0$. The main result of \cite{fpp2} in fact gives a proof of a stronger result than Theorem 1.1 of \cite{fpfree}, as the hypothesis that $f$ not be exceptional may be weakened to $f$ not being a Chebyshev polynomial. The results of \cite{fpp2} also hold over ground fields more general than $\C$, and for certain non-polynomial rational functions. 

It remains an open question to determine whether the hypotheses of Theorem 1.5 (or Corollary 1.6) of \cite{fpfree} are sufficient to imply $\FPP(G) = 0$. 
\begin{question} \label{quest}
Suppose that $G \leq \Aut(X^*)$ is a finitely-generated, self-similar, contracting subgroup such that each element of $g$ fixes no ends or infinitely many ends. Must $\FPP(G) = 0$? 
\end{question}
One approach would be to show that $\{g \in G_\infty  : \text{$g$ fixes at least one end of $X^*$}\}$ may be decomposed as the union of a set of measure zero and a set with non-empty interior. Then, under the hypotheses of Theorem 1.5 of \cite{fpfree}, \eqref{newcond} would show that $\FPP(G) = 0$. In many, and perhaps all, cases where they are known to occur, positive-measure subsets of $G_\infty$ fixing at least one end of $X^*$ arise as cosets of finite-index subgroups, and hence have non-empty interior. These include the iterated monodromy groups of Chebyshev polynomials \cite[Proposition 1.2]{fpfree}, profinite arithmetic iterated monodromy groups of certain Latt\`es maps \cite[Appendix A]{BJKL} and unicritical polynomials of odd degree \cite[Corollary 7]{fppe}, \cite[Theorem 1]{fpppos}, as well as some abstract groups \cite{fpppos2}, \cite[Theorem 9]{fppe}, \cite[Theorem 2]{fpppos}.

The basic idea in \cite{fpfree} of using contracting properties of the group $G$ can in fact be used to deduce that $\FPP(G) = 0$ in at least some cases, as we now show. Following Proposition \ref{lemma: FPP 0 three conditions} below we discuss the applicability of these methods to iterated monodromy groups of post-critically finite complex polynomials, and in particular note that they do not apply to every such group. 

In what follows, we adopt some of the notation of \cite{fpp2}. Denote by $X^n$ the truncated tree of height $n$, $\mathcal{L}_n \subset X^n$ the set of vertices of height $n$, and $\pi_n$ the natural projection from $G \leq \Aut(X^*)$ to $\Aut(X^n)$. For $a \in \pi_n(G)$, define the cone set $C_a = \pi_n^{-1}(a)$. The tree $X^*$ itself gives rise to an action $\mathcal{T}$ on $G$ by taking sections: $\mathcal{T}_v(g) = g|_v$ for $v \in X^*$. See \cite[Section 3]{fpp2} for details.   

\begin{lemma}[{Subindependence lemma}]
\label{lemma: subindependence lemma}
Let $G\le \mathrm{Aut}(X^*)$ be a closed self-similar subgroup with Haar probability measure $\mu$. Then, for every $n,m \geq 1$, $a \in \pi_n(G)$, $b \in \pi_m(G)$, and $v \in \mathcal{L}_n$ satisfying $C_a \cap \mathcal{T}_v^{-1}(C_b) \neq \emptyset,$ we have $$\mu(C_a \cap \mathcal{T}_v^{-1}(C_b)) \geq \mu(C_a) \cdot \mu(C_b).$$
\end{lemma}
\begin{proof}
The proof follows a similar idea as \cite[Proposition 3.2]{fpp2}. Denote $\mathrm{id}_n = \pi_n(\mathrm{id})$ and $\mathrm{id}_m = \pi_m(\mathrm{id})$. As $C_a \cap \mathcal{T}_v^{-1}(C_b) \neq \emptyset$ and $\mu$ is translation invariant, we get 
$$\mu(C_a \cap \mathcal{T}_v^{-1}(C_b)) = \mu(C_{\mathrm{id}_n} \cap \mathcal{T}_v^{-1}(C_{\mathrm{id}_m})).$$
Recall that $\mathrm{St}_G(n)$ is the kernel of $\pi_n$. Let $\varphi: \pi_{n+m}(\mathrm{St}_G(n)) \rightarrow \pi_m(G)$ be the group homomorphism induced by $\varphi_v^m \mid_{\mathrm{St}_G(n)}$, that is, by taking $\pi_m$ of the section at $v$. Then 
\begin{align*}
\mu(C_a \cap \mathcal{T}_v^{-1}(C_b)) &=\mu(C_{\mathrm{id}_n} \cap \mathcal{T}_v^{-1}(C_{\mathrm{id}_m}))= \frac{|\ker(\varphi)|}{|G: \mathrm{St}_G(n+m)|} \\
&= \frac{1}{|\mathrm{Im}(\varphi)|}\cdot  \frac{|\pi_{n+m}(\mathrm{St}_G(n))|}{|G: \mathrm{St}_G(n+m)|} \geq \frac{1}{|\pi_m(G)|} \cdot\frac{|\pi_{n+m}(\mathrm{St}_G(n))|}{|G: \mathrm{St}_G(n+m)|} \\
&= \frac{1}{|\pi_m(G)|}\cdot  \frac{1}{|\pi_n(G)|}= \mu(C_a) \cdot \mu(C_b). \qedhere
\end{align*}
\end{proof}

For $G \leq \Aut(X^*)$ and $v \in X^*$, denote by $G_v$ the set of all sections at $v$ of elements of $G$ fixing $v$. As in \cite[Section 2.3]{fpp2}, we say that $G \leq \Aut(X^*)$ is virtually super strongly fractal if $G_v = G$ for all $v \in X^*$ and the group
$$
K_G := \bigcap_{v \in \gamma} \textrm{St}_G(|v|)_v
$$
has finite index in $G$. Here $\gamma$ is an infinite rooted path in $X^*$ and $|v|$ is the distance of $v$ from the root of $X^*$. The group $K_G$ is independent of the choice of $\gamma$.
\begin{lemma}
Let $G \leq \mathrm{Aut}(X^*)$ be a closed virtually super strongly fractal group. Assume further that there exists $S$, a set of representatives of $G/K_G$, such that every $s\in S$ fixes infinitely many ends of $X^*$. Then $\mathrm{FPP}(G) = 0$.
\label{lemma: FPP 0 best result}
\end{lemma}
\begin{proof}
Let $Y_n : G \to \mathbb{Z}_{\geq 0}$ be the number of fixed points of of $\pi_n(g)$ in $\mathcal{L}_n$. Following the idea of \cite[Theorem 4.5]{fpp2}, by the martingale strategy it is enough to prove that given $r > 0$, there exist $\epsilon > 0$ and $m \geq 1$ depending only on $r$, such that $$\mu(Y_{n+m} > r\mid Y_n = r) \geq \epsilon$$ for all $n \geq 1$. 

Let $m \geq 1$ be such that $Y_m(s) > r$ for every $s \in S$, which exists as $S$ is a finite set. Let $A_{n,r} := \{a \in \pi_n(G): Y_n(a) = r\}$. Then 
\begin{align*}
\mu(Y_{n+m)} > r\mid  Y_n = r ) &= \frac{\sum_{a \in A_{n,r}} \mu(C_a \cap Y_{n+m} > r)}{\mu(C_{A_{n,r}})}.
\end{align*}

Let $v_a \in \mathcal{L}_n$ be a vertex fixed by $a$. Observe that there exists $g_a \in C_a$ such that $g_a|_{v_a} = s$ for some $s \in S$. Indeed, take any $g_1 \in C_a$. Then, $g_1|_{v_a}$ is in some coset $G/K_G$. Multiplying $g_1$ by a suitable element $g_2 \in \mathrm{St}_G(n)$ with $g_2|_{v_a} \in K_G$, we obtain $g_a$. Therefore, 
\begin{align*}
\mu(Y_{n+m)} > r\mid  Y_n = r ) &\geq \frac{\sum_{a \in A_{n,r}} \mu(C_a \cap \mathcal{T}_{v_a}^{-1}(C_s))}{\mu(C_{A_{n,r}})} \\
&\geq \frac{\sum_{a \in A_{n,r}} \mu(C_a)\cdot \mu(C_s)}{\mu(C_{A_{n,r}})} \\
&= \mu(C_s)= |\pi_m(G)|^{-1},
\end{align*}
where in the second step we use Lemma \ref{lemma: subindependence lemma}. The proof follows by taking $\epsilon(r) := |\pi_m(G)|^{-1}$.
\end{proof}

For the next proposition, $G$ will be a contracting group generated by the states of a kneading automaton. See \cite[Sections 2.11 and 6.7]{ssgps} for definitions. This class of groups includes the iterated monodromy groups of all post-critically finite complex polynomials \cite[Theorem 6.8.3]{ssgps}, and in addition such groups are contracting. Let $\mathcal{N}(G)$ denote the nucleus of $G$. 

\begin{proposition}
Let $G=\langle g_1,\dotsc,g_n\rangle \leq \Aut(X^*)$ be a contracting, virtually super strongly fractal group generated by the states $g_1, \ldots, g_n$ of a kneading automaton. Assume further that:
\begin{enumerate}
\item there exists $x_0 \in X$ moved only by one generator $g_i$;
\item for all $j \neq i$, we have $g_j \mid_{x_0} \neq g_i$; 
\item every element in $\mathcal{N}(G) \cap \langle g_j:j\ne i\rangle$ fixes infinitely many ends of $X^*$.
\end{enumerate}
Then, $\mathrm{FPP}(G) = 0$.
\label{lemma: FPP 0 three conditions}
\end{proposition}

\begin{proof}
By hypothesis, any $g\in \langle g_j:j\ne i\rangle$ fixes $x_0$. By condition (ii) we have 
$$g|_{x_0}\in \langle g_j: j\ne i\rangle.$$
In particular $g|_{x_0}$ fixes $x_0$ too by condition (i). As $G$ is contracting, there exists $n\ge 1$ such that $g|_{x_0^n}$ is in $\mathcal{N}(G) \cap \langle g_j:j\ne i\rangle$. Then, this section fixes infinitely many ends by condition (iii), i.e. $g$ fixes infinitely many ends.

Because $\{g_1, \ldots, g_n\}$ is the set of states of a kneading automaton, we have $g_1 \cdots g_n \in K_G$ for some ordering of the generators. Therefore, a set of representatives of $G/K_G$ is contained in $\langle g_j\mid j\ne i\rangle$ and they all fix infinitely many ends by (iii). Now, as we can take a set of representatives for $\overline{G}/{\overline{K}_G}$ from $G/K_G$, the assumptions in Lemma \ref{lemma: FPP 0 best result} are satisfied, so we get $\mathrm{FPP}(G) = 0$.
\end{proof}
We observe that there exist examples of post-critically finite complex polynomials whose iterated monodromy groups are generated by the states of a kneading automaton that fails to satisfy conditions (1) and (2) of Proposition 4. Such examples do not appear to be typical; the smallest one we found was a polynomial of degree $8$. 

Another obstacle to applying Proposition \ref{lemma: FPP 0 three conditions} to iterated monodromy groups is the issue of whether such groups are virtually super strongly fractal. The examples we have found support a positive answer, but we have been unable to find a proof. We state the following question for post-critically finite rational functions,
where we view $\text{IMG}(f)$ as a subgroup of $\Aut(X^*)$ via a standard action \cite[Section 5.2]{ssgps}. 

\begin{question}
Let $f \in \C(x)$ be post-critically finite of degree $d \geq 2$. Must $\textrm{IMG}(f)$ be virtually super strongly fractal?
\end{question}

\end{document}